\documentclass[11pt,oneside,a4paper]{article}
\usepackage{blindtext}
\usepackage{amsmath,amsfonts,amssymb,amsthm,eucal}
\usepackage{microtype}
\usepackage[colorlinks, citecolor=blue, linkcolor=blue, urlcolor=Maroon, filecolor=Maroon, linktocpage=true]{hyperref}
\usepackage[usenames,dvipsnames,svgnames,table]{xcolor}
\usepackage[normalem]{ulem}
\usepackage{lmodern}
\usepackage{enumerate}
\usepackage{bbm}
\usepackage{tikz}
\usepackage{tikz-cd}
\usepackage[font=footnotesize,labelfont=bf]{caption}
\usepackage{subcaption}
\usepackage{framed}
\usepackage[nottoc,notlot,notlof]{tocbibind}
\usepackage{tocloft}
\usepackage{morefloats}
\usepackage{float}
\usepackage[nottoc,notlot,notlof]{tocbibind}
\usepackage{tocloft}
\usepackage{hhline}
\usepackage{varwidth}

\newcommand{\D}{\,\mathrm{d}}
\theoremstyle{definition}
\newtheorem{Th}{Theorem}[section]

\newtheorem{Prop}[Th]{Proposition}

\newtheorem{Lem}[Th]{Lemma}
\newtheorem{Def}[Th]{Definition}

\theoremstyle{definition}
\newtheorem{rem}[Th]{Remark}
\usepackage[a4paper,
	hcentering,
	text={453pt,686pt},
]{geometry}
\allowdisplaybreaks
\catcode`@=11
\@addtoreset{equation}{section}
\catcode`@=12
\begin{document}
%
\begin{center}\noindent
\textbf{\Large Special homogeneous curves}\\[2em]
{\fontseries{m}\fontfamily{cmss}\selectfont \large David\ Lindemann}\\[1em] 
{\small
Department of Mathematics, Aarhus University\\
Ny Munkegade 118, Bldg 1530, DK-8000 Aarhus C, Denmark\\
\texttt{david.lindemann@math.au.dk}
}
\end{center}
\vspace{1em}
\begin{abstract}
\noindent
We classify all special homogeneous curves. A special homogeneous curve $\mathcal{H}$ consists of connected components of the hyperbolic points in the level set $\{h=1\}$ of a homogeneous polynomial $h$ in two real variables of degree at least three, and admits a transitive group action of a subgroup $G\subset\mathrm{GL}(2)$ on $\mathcal{H}$ that acts via linear coordinate change.
\end{abstract}
\tableofcontents
\section{Introduction and main results}
	In this work we classify a certain class of homogeneous real curves. We study hyperbolic homogeneous polynomials $h:\mathbb{R}^2\to\mathbb{R}$ of degree $\tau\geq 3$, so that the action of their respective linear automorphism group $G^h\subset\mathrm{GL}(2)$ acts transitively on at least one connected component of $\{h=1\}\cap\{\text{hyperbolic points of }h\}$. Hyperbolicity means that there exists a point $p\in\{h>0\}$, such that $-\partial^2h_p$ is of Minkowski type. In dimension two, $-\partial^2 h_p$ is required to have precisely one positive and one negative eigenvalue. Two polynomials are called equivalent if they are related by a linear transformation, and similarly two connected components of $\{h=1\}\cap\{\text{hyperbolic points of }h\}$ for a given polynomial $h$ are called equivalent if they are related by a linear transformation. Throughout this work we will refer to such homogeneous curves as \textit{special homogeneous curves}. The reason for the term \textit{special} will become clear momentarily. Note that for $\tau=2$ it is a well-known fact that there exists precisely one such polynomial up to equivalence, namely $h=x^2-y^2$. The level set $\{h=1\}$ is the two-sheeted hyperbola and  consists of two connected components that exclusively contain hyperbolic points and are related by a reflection. For $\tau=3$ and $\tau\geq 4$, we are in the realm of projective special real (PSR) curves and generalized projective special real (GPSR) curves, respectively \cite{CHM,L2}. In higher dimensions, PSR and GPSR manifolds are defined analogously. PSR curves have been completely classified in \cite{CHM} and there exists up to equivalence precisely one hyperbolic homogeneous cubic polynomial $h= x^2y$, such that $\{h=1\}\cap\{\text{hyperbolic points of }h\}$ is a homogeneous space. Similar to the case $\tau=2$, $\{x^2y=1\}$ contains only hyperbolic points and the two connected components are also related by a reflection. For $\tau\geq 4$, that is GPSR curves, there is a classification of all hyperbolic homogeneous quartic curves and their respective connected components of $\{h=1\}\cap\{\text{hyperbolic points of }h\}$ \cite{L4}. Up to equivalence there exist exactly two hyperbolic homogeneous quartic polynomials
		\begin{equation*}\label{eqn_hom_quartic_polys}
			h_1= x^4-x^2y^2 + \frac{1}{4}y^4,\quad h_2= x^4-x^2y^2+\frac{2\sqrt{2}}{3\sqrt{3}}xy^3-\frac{1}{12}y^4,
		\end{equation*}
	such that the respective positive level sets contain a homogeneous curve. The set $\{h_1=1\}$ contains four equivalent connected components and $\{h_1=1\}$ contains two equivalent connected components, all of which consisting only of hyperbolic points. For degree $\tau\geq 5$ there is no known general classification, and to our knowledge there has also not been any specific case study. Observe that if $\{h=1\}$ contains a special homogeneous curve $\mathcal{H}$, then $\mathcal{H}$, by being a Riemannian homogeneous space with respect to the centro-affine fundamental form $g=-\partial^2 h|_{T\mathcal{H}\times T\mathcal{H}}$, is closed in the ambient space $\mathbb{R}^2$ \cite[Prop.\,1.8]{CNS} and consists thereby automatically of one or more connected components of $\{h=1\}$.
	
	Our main result is the following complete classification of special homogeneous curves.
	
	\begin{Th}\label{thm_special_hom_curves}
		Let $h:\mathbb{R}^2\to\mathbb{R}$ be a homogeneous polynomial of degree $\tau\geq 3$, such that $\{h=1\}$ contains a special homogeneous curve. Then $h$ is equivalent to
			\begin{equation*}
				h=x^{\tau-k}y^k
			\end{equation*}
		for precisely one $k\in\left\{1,\ldots, \lfloor \frac{\tau}{2}\rfloor\right\}$. The level set $\{h=1\}$ has two equivalent connected components if $k$ is odd or $\tau$ is odd and $k$ is even, and four equivalent connected components if both $\tau$ and $k$ are even. If $\tau$ is odd, $G^h\cong\mathbb{R}\times\mathbb{Z}_2$ with the $\mathbb{Z}_2$-factor acting via $x\to -x$ for $k$ odd and via $y\to -y$ for $k$ even. If $\tau$ is even and $k$ odd, $G^h\cong \mathbb{R}\times\mathbb{Z}_2$ with the $\mathbb{Z}_2$-factor acting via $x\to -x$. If $\tau$ is even, $k$ is even, and $2k\ne\tau$, $G^h\cong\mathbb{R}\times\mathbb{Z}_2\times\mathbb{Z}_2$ with the $\mathbb{Z}_2$-factors acting via $x\to-x$ and $y\to-y$, respectively. Lastly, if $\tau$ is even, $k$ is even, and $\tau=2k$, $G^h\cong(\mathbb{R}\times\mathbb{Z}_2\times\mathbb{Z}_2)\ltimes\mathbb{Z}_2$, with the first three factors acting as in the previous case and the last $\mathbb{Z}_2$-factor acting via $(x,y)\to(y,x)$. For fixed degree $\tau$ of the defining polynomial, there exist precisely $\left\lfloor \frac{\tau}{2}\right\rfloor$ inequivalent special homogeneous curves.
	\end{Th}
	
	Note that a complete classification of homogeneous PSR manifolds has been found in \cite{DV}. For higher-degree polynomials, no general classification result is known. We expect that a most general classification result without restricting either the degree to a specific one, or restricting the dimension, will be very difficult to obtain with currently available mathematical tools.
	
	The main goal of this work is to present a first step in a possible classification of a larger class of higher-dimensional \textit{special homogeneous spaces}, that is homogeneous GPSR manifolds in arbitrary dimension. This is motivated by the classical and nearly completely open problem of classifying homogeneous polynomials in an arbitrary number of real variables. Even when restricting to hyperbolic polynomials, there are only comparatively few known results. PSR curves and surfaces and their defining polynomials have been classified in \cite{CHM} and \cite{CDL}, respectively. PSR manifolds of dimension $n\geq3$ with reducible defining polynomial have been classified in \cite{CDJL}, and homogeneous PSR manifolds have been completely classified in \cite{DV}. For higher degree polynomials, classification results are scarce. There is a classification of hyperbolic quartic hyperbolic homogeneous polynomials in two real variables \cite{L4}, and the difficulty and work needed to obtain this result in comparison with the PSR curves classification was already significantly higher than initially expected. To our knowledge, there is no classification result of hyperbolic quartics in higher dimensions. However, in an extensive work by A.B. Korchagin and D.A. Weinberg \cite{KW} the authors successfully classify all isotopy types of affine and projective quartic curves. For degree of the polynomials greater than four, there are no classification results. Hoping to get to, even in the long term, a complete classification of hyperbolic homogeneous polynomials in arbitrary degree is most likely not realistic. Even in the cubic case, the corresponding moduli space is only slightly better understood when restricting the global geometry of the corresponding PSR manifolds \cite{L2,L3}. We aim nonetheless for a classification of all special homogeneous spaces. The next step after this work is the classification of special homogeneous surfaces. We are currently cautiously optimistic to obtain such a result in a reasonable amount of time.
	
	The topics treated in this work are additionally motivated by special K\"ahler geometry \cite{F,ACD} and by the study of the geometry of K\"ahler cones \cite{DP,W,M}. In special K\"ahler geometry, affine and projective special K\"ahler manifolds are studied. The so-called \textbf{supergravity r-map} allows one to explicitly construct such manifolds from connected special real manifolds, and this construction in particular preserves geodesic completeness \cite{CHM}. But for hyperbolic homogeneous polynomials of degree at least four, no straightforward generalisation of that construction exists. One ansatz to find a good candidate is to study what it should do with special homogeneous spaces, as it might then be of more algebraic nature and should be required to preserve the symmetries in some sense. In the geometry of K\"ahler cones, hyperbolic homogeneous polynomials and their positive level sets appear as follows. Given a compact K\"ahler manifold $X$ of complex dimension $\tau\geq 3$, the real homogeneous polynomial
		\begin{equation*}
			h:H^{1,1}(X,\mathbb{R})\to\mathbb{R},\quad [\omega]\mapsto\int_X \omega^\tau
		\end{equation*}
	is hyperbolic. In particular all points in the \textbf{K\"ahler cone} $\mathcal{K}\subset H^{1,1}(X,\mathbb{R})$, that is the subset of classes in $H^{1,1}(X,\mathbb{R})$ containing a K\"ahler metric, are hyperbolic points of $h$ by the Hodge-Riemann bilinear relations, so $\mathcal{H}=\{h=1\}\cap\mathcal{K}$ is a GPSR manifold. It is in general not known which hyperbolic polynomials can be constructed in this way. A reasonable ansatz would be to try and find a compact K\"ahler manifold leading to polynomials corresponding to special homogeneous curves. We are currently working on that problem and hope to extend possible results to special homogeneous surfaces once they have been successfully classified.
	
	\paragraph*{Acknowledgements} This work was partly supported by the \textit{German Research Foundation} (DFG) under the RTG 1670 ``Mathematics Inspired by String Theory and Quantum Field Theory'', and partly by a ``Walter Benjamin PostDoc Fellowship'' granted to the author by the DFG. The author would like to thank \'Aron Szab\'o for helpful discussions during the initial approach to this project, and Andrew Swann for suggesting that certain standard forms of polynomials introduced in \cite{L2} were not necessarily the most realistic ansatz for this work.
\section{Preliminaries}
	We start by giving rigorous definitions of the considered objects and technical tools needed to prove our results.

	\begin{Def}
		A homogeneous polynomial $h:\mathbb{R}^{n+1}\to\mathbb{R}$ is called \textbf{hyperbolic} if there exists $p\in\{h>0\}$, such that $-\partial^2h_p$ is of Minkowski type. Such a point $p$ is called \textbf{hyperbolic point} of $h$. Two homogeneous polynomials $h,\overline{h}$ are called equivalent if they are related by a linear coordinate change, i.e. there exists $A\in\mathrm{GL}(n+1)$, such that $A^*h=\overline{h}$. A hypersurface $\mathcal{H}$ contained in the level set $\{h=1\}$ of a hyperbolic homogeneous polynomials of degree $\tau\geq 3$ is called \textbf{projective special real (PSR) manifold} for $\tau=3$ and \textbf{generalised projective special real (GPSR) manifold} for $\tau\geq 4$. Two (G)PSR manifolds $\mathcal{H}$ and $\overline{\mathcal{H}}$ are called equivalent if there exists $A\in\mathrm{GL}(n+1)$, such that $A(\overline{\mathcal{H}})=\mathcal{H}$.
	\end{Def}
	
	As in the introduction, we will refer to homogeneous (G)PSR manifolds as \textbf{special homogeneous spaces}, respectively \textbf{special homogeneous curves} in the one-dimensional case.
	
	\begin{rem}
		For two equivalent (G)PSR manifolds $\mathcal{H}\subset\{h=1\}$ and $\overline{\mathcal{H}}\subset\{\overline{h}=1\}$ with $A(\overline{\mathcal{H}})=\mathcal{H}$ for some $A\in\mathrm{GL}(n+1)$, the corresponding defining polynomials are automatically equivalent via $A^*h=\overline{h}$. The converse does in general not hold true, which can be seen by restricting a given (G)PSR manifold to an open subset that is not the entire (G)PSR manifold.
	\end{rem}
	
	Euler's homogeneous function theorem implies that $\D h_p(p)=\tau h(p)$ for all homogeneous polynomials of degree $\tau\geq 3$. Hence, the position vector field $\xi\in\mathfrak{X}(\mathbb{R}^{n+1})$ is transversal to any given (G)PSR manifold. We can thus study (G)PSR manifolds in the setting of centro-affine geometry. The \textbf{centro-affine fundamental form} $g$ of a centro-affine hypersurface $\mathcal{H}$ is defined to be the unique symmetric $(0,2)$-tensor fulfilling the \textbf{centro-affine Gau{\ss} equation}
		\begin{equation*}
			\mathrm{D}_X Y = \nabla^{\mathrm{ca}}_X Y + g(X,Y)\xi
		\end{equation*}
	for all $X,Y\in\mathfrak{X}(\mathcal{H})$. In the above equation, $\mathrm{D}$ denotes the canonical flat connection on $\mathbb{R}^{n+1}$, and $\nabla^{\mathrm{ca}}$ denotes the \textbf{centro-affine connection} which is also uniquely determined by the centro-affine Gau{\ss} equation. Note that the centro-affine connection and the Levi-Civita connection induced by $g$, assuming $g$ being non-degenerate, do in general not coincide. For a reference on affine and centro-affine geometry see \cite{NS}. In the case of (G)PSR manifolds $\mathcal{H}\subset\{h=1\}$, the centro-affine fundamental form is a Riemannian metric and of the form $g=-\frac{1}{\tau}\partial^2h|_{T\mathcal{H}\times T\mathcal{H}}$ \cite[Prop.\,1.3]{CNS}. The Riemannian property follows from the hyperbolicity condition. To see this, observe that for all $p\in\mathcal{H}$, $\mathrm{ker}(\!\D h_p)$ is orthogonal to $\xi_p$ with respect to the Lorentzian metric $-\partial^2h$ on $\mathbb{R}_{>0}\cdot\mathcal{H}$. This follows, again, from Euler's theorem for homogeneous functions. Since $-\partial^2h(\xi,\xi)=-\tau(\tau-1)h$ is negative on $\mathcal{H}$, we deduce that $g$ is indeed positive definite by the hyperbolicity of every point in $\mathcal{H}$.
	
	\begin{Def}
		For a given homogeneous polynomial $h:\mathbb{R}^{n+1}\to\mathbb{R}$ we denote by $G^h\subset\mathrm{GL}(n+1)$ the linear automorphism group of $h$.
	\end{Def}
	
	Note that for any (G)PSR manifold $\mathcal{H}\subset\{h=1\}$, elements in $G^h$ are isometries with respect to the centro-affine fundamental form.
	
	\begin{Def}\label{def_bdr_beh}
		Let $\mathcal{H}\subset\{h=1\}$ be a (G)PSR manifold that is closed in the ambient space $\mathbb{R}^{n+1}$, and denote $U=\mathbb{R}_{>0}\cdot\mathcal{H}$. Then $\mathcal{H}$ is said to have \textbf{regular boundary behaviour} if
			\begin{center}
				\begin{minipage}{0.9\textwidth}
					\begin{enumerate}[(i)]
						\item $\D h_p\ne 0$, \label{eqn_reg_bdr_beh_i}
						\item $\left.-\partial^2 h\right|_{T_p(\partial U\setminus\{0\})\times T_p(\partial U\setminus\{0\})}\geq 0$ and $\dim\ker\left(\left.-\partial^2 h\right|_{T_p(\partial U\setminus\{0\})\times T_p(\partial U\setminus\{0\})}\right)=1$,\label{eqn_reg_bdr_beh_ii}
					\end{enumerate}
				\end{minipage}
			\end{center}
		for all $p\in\partial U\setminus\{0\}$. If condition \hyperref[eqn_reg_bdr_beh_i]{(i)} is violated, independently of the validity of condition \hyperref[eqn_reg_bdr_beh_ii]{(ii)}, $\mathcal{H}$ is called \textbf{singular at infinity}.
	\end{Def}
	
	In the case of closed PSR manifolds, having regular boundary behaviour is equivalent to condition \hyperref[eqn_reg_bdr_beh_i]{(i)} in Definition \ref{def_bdr_beh}, cf. \cite[Thm.\,4.12]{L2}, that is closed PSR manifolds have regular boundary behaviour if and only if they are not singular at infinity. In the case of special homogeneous curves, we will prove the following.
	
	\begin{Lem}\label{lem_homcurves_sing_at_inf}
		Special homogeneous curves are singular at infinity and thereby have non-regular boundary behaviour.
	\end{Lem}
	
	For the proof of the above lemma see Section \ref{sect_proofs}.
	
	\begin{rem}
		Special homogeneous spaces are closed in the ambient space $\mathbb{R}^{n+1}$. This follows from their completeness as Riemannian manifolds with respect to their centro-affine fundamental form. In fact, any geodesically complete (G)PSR manifold $(\mathcal{H},g)$ is closed in the ambient space $\mathbb{R}^{n+1}$ \cite[Prop.\,1.8]{CNS}. Whether the converse statement holds also true in general is an open question for $\mathrm{deg}(h)\geq 4$. For PSR manifolds the converse does indeed hold true \cite[Thm.\,2.5]{CNS}. Furthermore, it also holds true for all GPSR manifolds with regular boundary behaviour \cite[Thm.\,1.18]{CNS}.
	\end{rem}
	
	For our classification we will bring all considered polynomials to a certain form. In order to do so, we need the following result.
	
	\begin{Prop}\label{prop_intersection}
		Let $\mathcal{H}$ be a closed connected (G)PSR manifold. Then for all $p\in\mathcal{H}$
			\begin{equation*}
				\left(\mathbb{R}_{>0}\cdot\mathcal{H}\right)\cap\left(p+T_p\mathcal{H}\right)\subset\mathbb{R}^{n+1}
			\end{equation*}
		is convex and precompact.
		\begin{proof}
			Follows from \cite[Cor.\,1.11,\,Lem.\,1.14]{CNS}.
		\end{proof}
	\end{Prop}
	
	Since special homogeneous curves are automatically closed, the above in particular applies to each of their connected components.
	
\section{Proofs of Theorem \ref{thm_special_hom_curves} and Lemma \ref{lem_homcurves_sing_at_inf}}\label{sect_proofs}
	\begin{proof}[Proof of Theorem \ref{thm_special_hom_curves}]
		Using Proposition \ref{prop_intersection} we obtain that $\{h=0\}$ contains at least two lines. Thus we can without loss of generality assume that $\{x=0\}\subset\{h=0\}$ and $\{y=0\}\subset\{h=0\}$. Using that special homogeneous curves are necessarily closed in their ambient space we might further assume that $h$ is hyperbolic on $\{x>0,\,y>0\}$ and that the connected component of $\{h=1\}\cap\{x>0,\,y>0\}$ is a special homogeneous curve. Then $h$ is of the form
			\begin{equation*}
				h=xyP,
			\end{equation*}
		where $P$ is a homogeneous polynomial in $x$ and $y$ of degree $\mathrm{deg}(h)-2$. Then there exists a connected one-dimensional subgroup $G$ of $\mathrm{GL}(2)$ acting transitively on $\{h=1\}\cap\{x>0,\,y>0\}$. Note that $G\subset G^h$. This means that the action of $G$ must leave both $\{x=0\}$ and $\{y=0\}$ invariant. Hence, $G$ has an infinitesimal generator of the form
			\begin{equation*}
				a=\left(\begin{array}{cc}
					r & 0 \\
					0 & s
				\end{array}\right)
			\end{equation*}
		for some $r,s\in\mathbb{R}$ with $rs\ne 0$. By writing $h$ in the form
			\begin{equation*}
				h=\sum\limits_{k=1}^{\tau-1} f_{k} x^{\tau-k}y^k,
			\end{equation*}
		$f_k\in\mathbb{R}$ for all $1\leq k\leq \tau-1$, we find that $a$ is an infinitesimal symmetry of $h$ if and only if
			\begin{equation*}
				\left(\sum\limits_{k=1}^{\tau-1} (\tau-k)f_{k} x^{\tau-k}y^k\right)r + \left(\sum\limits_{k=1}^{\tau-1} kf_{k} x^{\tau-k}y^{k}\right)s = 0.\label{eqn_inf_sym_h_curves}
			\end{equation*}
		The above is equivalent to $((\tau-k)r+ks)f_k=0$ for all $1\leq k \leq \tau-1$. Since $k\mapsto\frac{k-\tau}{k}$ is strictly monotonously increasing in $k>0$, we deduce that $f_k\ne 0$ for at most one $k\in\{1,\ldots,\tau-1\}$. If all $f_k$ vanish identically, $h\equiv 0$, which is not a hyperbolic polynomial, and by the assumption that $h$ is positive on $\{x>0,\,y>0\}$ we obtain after rescaling $x$ or $y$ with a positive factor if necessary that $h$ is equivalent to $h=x^{\tau-k}y^k$ for some $k\in\{1,\ldots,\tau-1\}$. After possibly switching variables, we can further assume that $k\in\left\{1,\ldots,\left\lfloor \frac{\tau}{2}\right\rfloor\right\}$. We now need to show that each such $h$ is actually hyperbolic on $\{x>0,\,y>0\}$. To do so, consider for $t\in[0,1]$
			\begin{align*}
				-\partial^2h_{\left(\begin{smallmatrix}t \\ 1-t\end{smallmatrix}\right)}&=-\left(\begin{array}{cc}
					(\tau-k)(\tau-k-1) t^{\tau-k-2}(1-t)^k & (\tau-k)kt^{\tau-k-1}(1-t)^{k-1}\\
					(\tau-k)k t^{\tau-k-1}(1-t)^{k-1} & k(k-1)t^{\tau-k}(1-t)^{k-2}
				\end{array}\right),\\
				\det\left(-\partial^2h_{\left(\begin{smallmatrix}t \\ 1-t\end{smallmatrix}\right)}\right) &=(\tau-k)k(1-\tau)t^{2(\tau-k-1)}(1-t)^{2(k-1)}.
			\end{align*}
		One sees that $\det\left(-\partial^2h_{\left(\begin{smallmatrix}t \\ 1-t\end{smallmatrix}\right)}\right)$ is negative for all $t\in(0,1)$, meaning by dimensional reason and homogeneity of $h$ that $-\partial^2h$ has one negative and one positive eigenvalue on $\{x>0,\,y>0\}$. By the positivity of $h$ on the latter set we have thus shown that for all $k\in\left\{1,\ldots,\left\lfloor \frac{\tau}{2}\right\rfloor\right\}$, every point in $\{x>0,\,y>0\}$ is a hyperbolic point of $h$. For $\tau$ and $k$ fixed, we find that a possible choice for the infinitesimal generator $a$ of $G$ is given by $r=k,\,s=k-\tau$.
		
		It remains to show that for two distinct allowed choices for $k$ the corresponding polynomials are not equivalent. The only linear transformations that leave $\{x=0\}$ and $\{y=0\}$ invariant are compositions of diagonal transformations and $\left(\begin{smallmatrix}0&1\\1&0\end{smallmatrix}\right)$. By the allowed range for $k$ we can exclude transformations switching $x$ and $y$, so that only diagonal transformations are allowed. Any such transformation acts on $x^{\tau-k}y^k$ via rescaling. Hence, two polynomials of the form $h=x^{\tau-k}y^k$ cannot be equivalent for different choices of $k$. The number of inequivalent connected special homogeneous curves with $\mathrm{deg}(h)=\tau$ is thus $\left\lfloor \frac{\tau}{2}\right\rfloor$.
		
		Next, we will determine the number of connected components of each set $\{h=1\}$. For $h=x^{\tau-k}y^k$, $k\in\left\{1,\ldots,\left\lfloor \frac{\tau}{2}\right\rfloor\right\}$, observe that $\{h=1\}$ has two connected components if either $\tau$ is odd, or $\tau$ is even and $k$ is odd. If $\tau$ is even and $k$ is even, $h$ is non-negative and $\{h=1\}$ has four connected components. See Figure \ref{fig_examples_ccs} for examples of plots of $\{h=1\}$ in each case.
			\begin{figure}[H]
			    \centering
			    \begin{subfigure}[b]{0.23\textwidth}
			        \includegraphics[width=\textwidth]{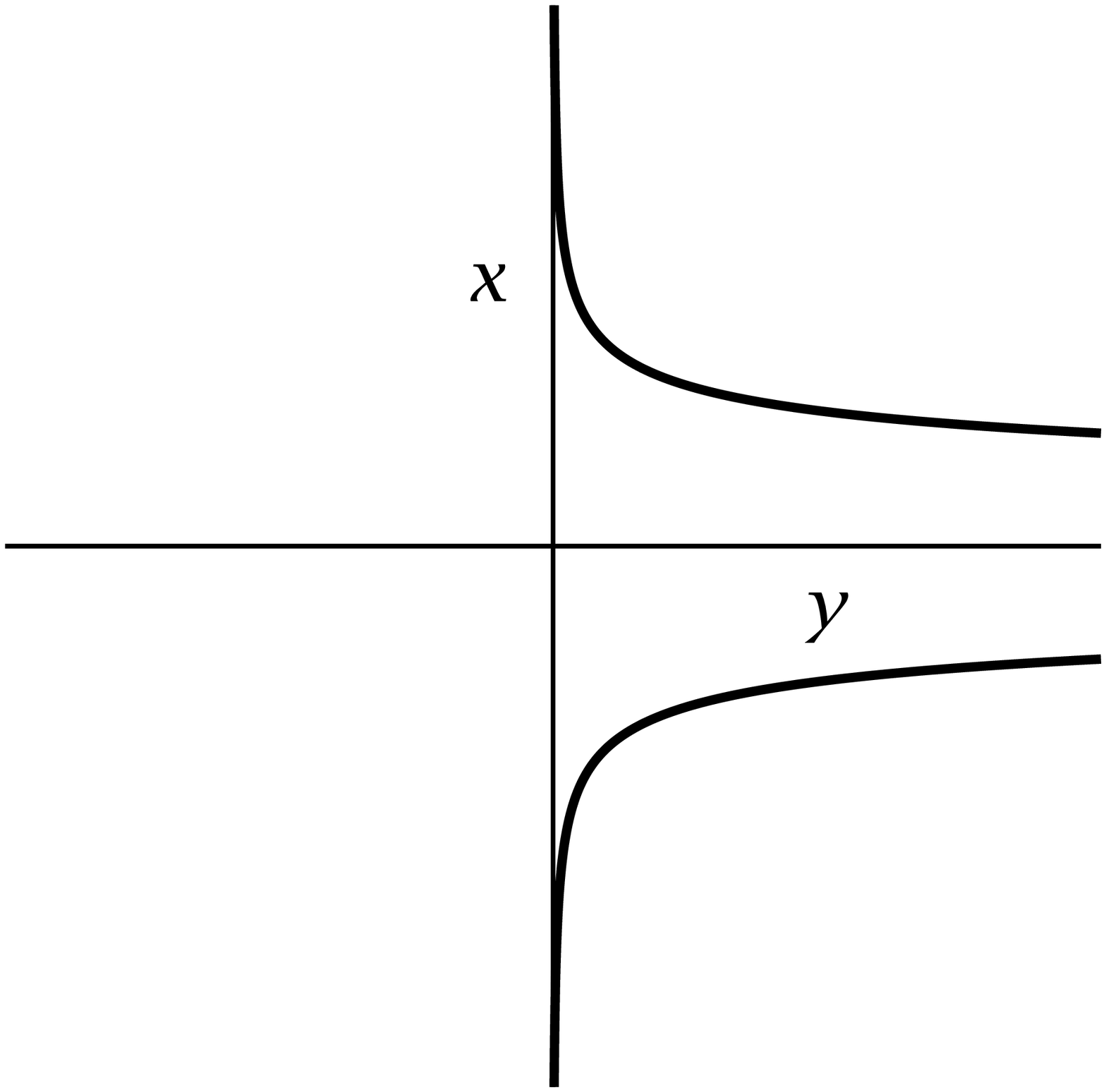}
			        \caption{$\tau=5$, $k=1$}
			        \label{fig_examples_ccs_a}
			    \end{subfigure}
			    \
			    \begin{subfigure}[b]{0.23\textwidth}
			        \includegraphics[width=\textwidth]{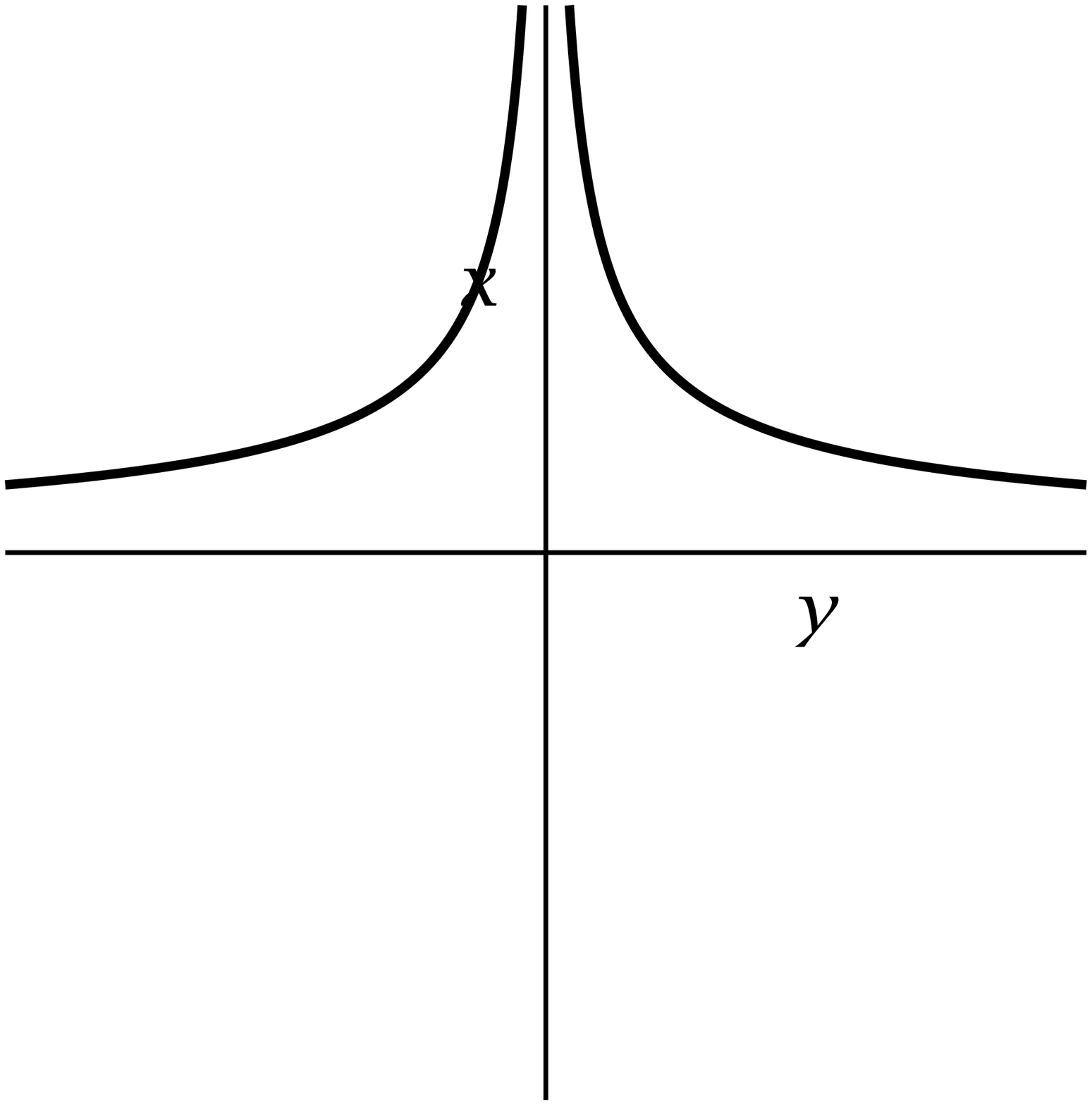}
			        \caption{$\tau=5$, $k=2$}
			        \label{fig_examples_ccs_b}
			    \end{subfigure}
			    \
			    \begin{subfigure}[b]{0.23\textwidth}
			        \includegraphics[width=\textwidth]{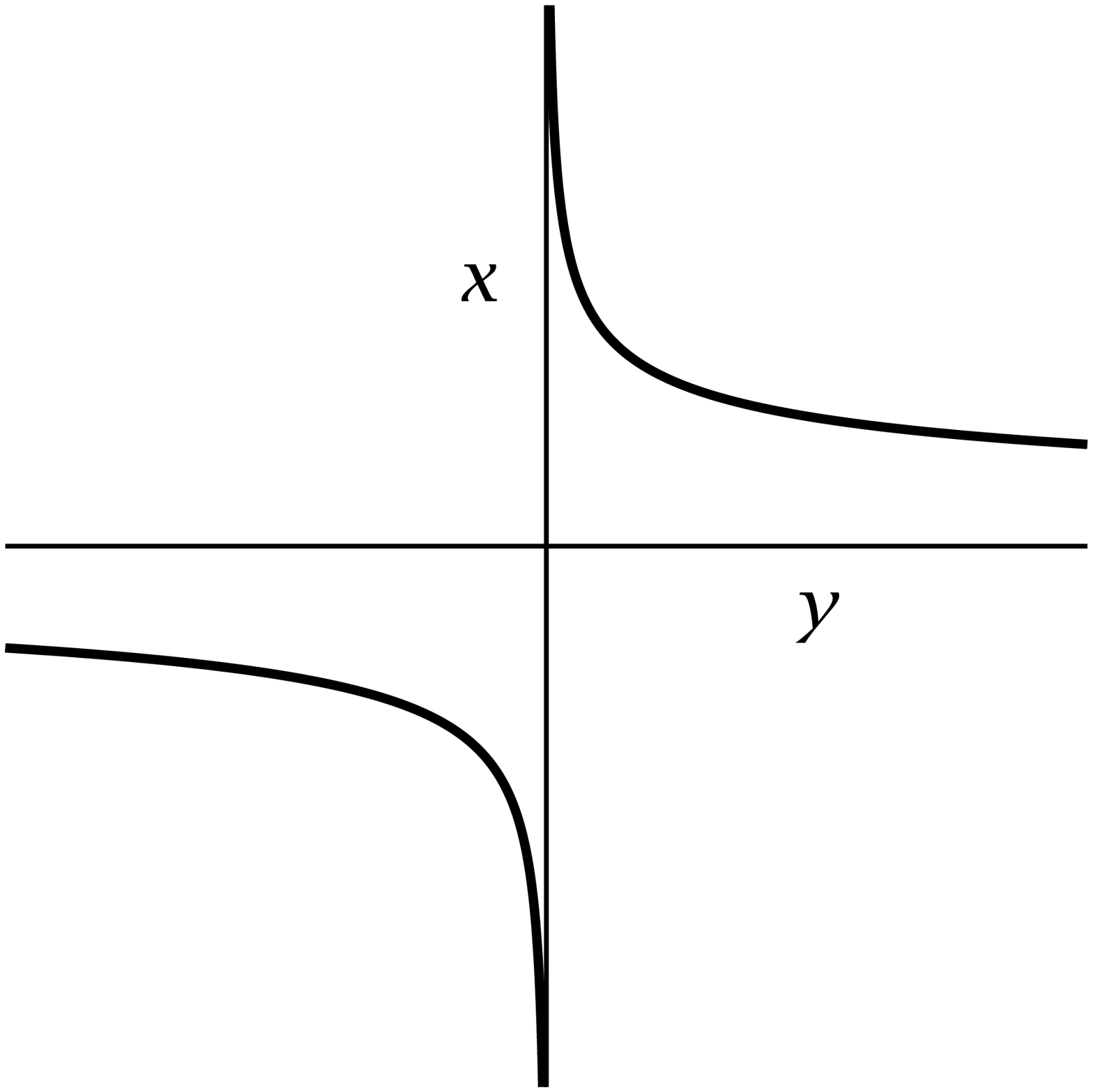}
			        \caption{$\tau=4$, $k=1$}
			        \label{fig_examples_ccs_c}
			    \end{subfigure}
			    \
			    \begin{subfigure}[b]{0.23\textwidth}
			        \includegraphics[width=\textwidth]{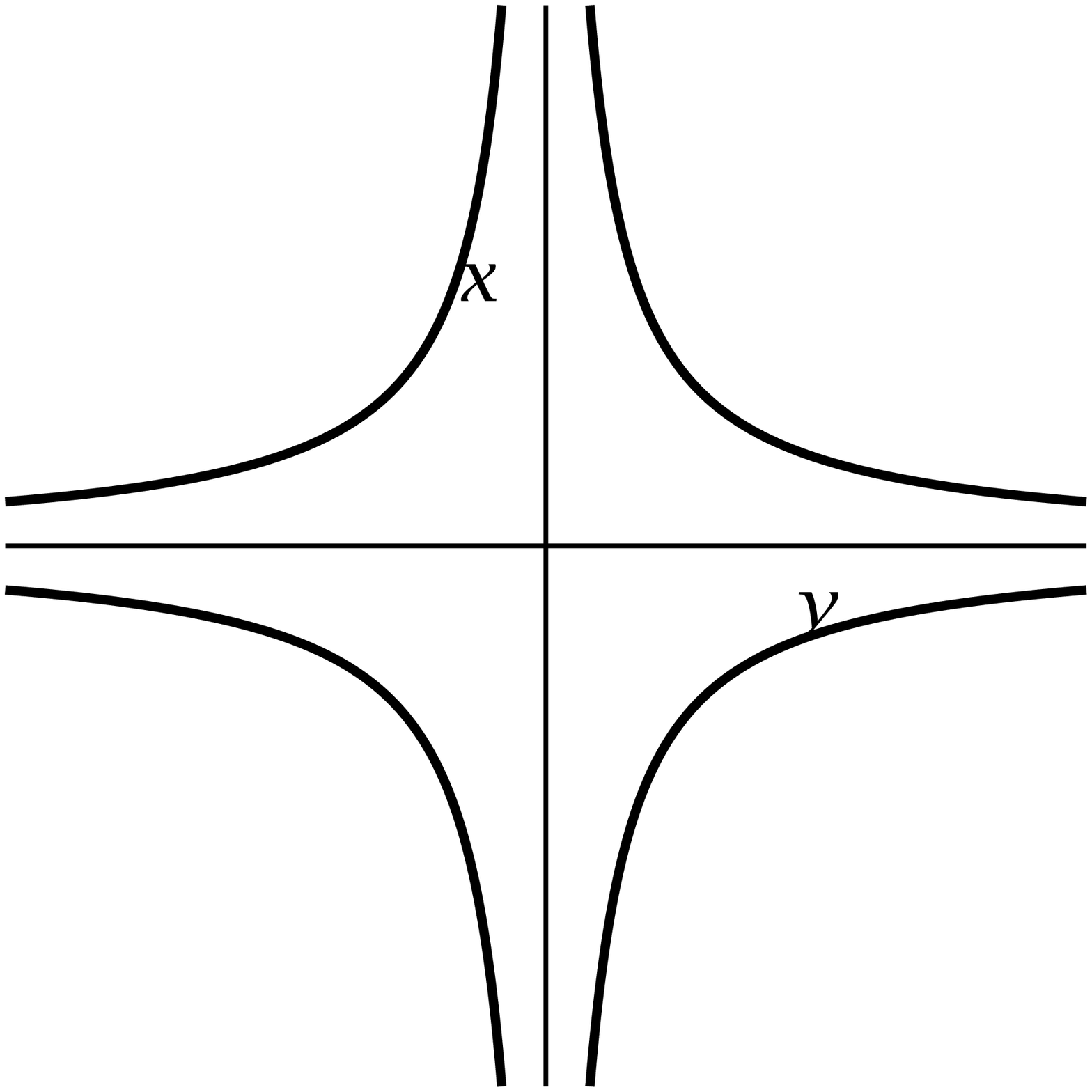}
			        \caption{$\tau=4$, $k=2$}
			        \label{fig_examples_ccs_d}
			    \end{subfigure}
			    \caption{$\{h=1\}$ for different values of $\tau$ and $k$.}\label{fig_examples_ccs}
			\end{figure}
		\noindent
		One quickly checks that the connected components of $\{h=1\}$ are equivalent independently of the choice of $\tau$ and $k$ via combinations of $x\to-x$, $y\to-y$ and $x\leftrightarrow y$, and thus in particular contain exclusively hyperbolic points. In order to determine the automorphism groups $G^h$, the only case that one has to be careful with is $k$ even and $\tau=2k$. In that case we obtain an additional symmetry given by $(x,y)\to(-x,-y)$ which does not commute with e.g. $x\to-x$. Hence, in these cases, $G^h$ is a semi-direct product of $\mathbb{R}\times\mathbb{Z}_2\times\mathbb{Z}_2$ and $\mathbb{Z}_2$, the last factor acting by switching $x$ and $y$.
	\end{proof}
	
	\begin{rem}
		As mentioned in the introduction, it is in general an open and most likely very difficult problem to classify all special homogeneous spaces of dimension higher than one. For special homogeneous curves we have used that the corresponding sets $\{h=0\}$ are easy to control. For higher dimensional spaces it is however much more difficult to control the real projective algebraic varieties $\{h=0\}$, even in the comparatively most likely easiest of the open cases, $\mathrm{deg}(h)=4$. Note however that the classification of homogeneous PSR manifolds of any dimension in \cite{DV} did not require complete control over $\{h=0\}$.
	\end{rem}
	
	\begin{proof}[Proof of Lemma \ref{lem_homcurves_sing_at_inf}]
		For any $h=x^{\tau-k}y^k$ as in Theorem \ref{thm_special_hom_curves}, we have $\D h = (\tau-k) x^{\tau-k-1}y^k \D x + kx^{\tau-k}y^{k-1} \D y$. One of the connected components of $\{h=1\}$ is contained in, and spans, $\{x>0,\,y>0\}$. By $\tau-k\geq 2$, $\D h$ vanishes identically on $\{x=0\}$. This shows that every special homogeneous curve is singular at infinity.
	\end{proof}

\end{document}